\documentclass[12pt]{amsart}
\usepackage{amscd}
\usepackage[latin1]{inputenc}
\usepackage{amsmath}
\usepackage{amssymb}
\usepackage{amsthm}
\usepackage{latexsym}
\usepackage{pstcol,pst-plot,pst-3d}
\usepackage{mathptmx}
\usepackage{multicol}
\usepackage{hyperref}

\psset{unit=0.7cm,linewidth=0.8pt,arrowsize=2.5pt 4}
\newpsstyle{fatline}{linewidth=1.5pt}
\newpsstyle{fyp}{fillstyle=solid,fillcolor=verylight}
\definecolor{verylight}{gray}{0.97}
\definecolor{light}{gray}{0.9}
\definecolor{medium}{gray}{0.85}

\def\frk{\frak}               % font for "Fraktur"

\def\mm{{\frk m}}

\def\Phi{{\frk n}}
\def\Phi{{\frk N}}
\def\opn#1#2{\def#1{\operatorname{#2}}} % to make operators
\opn\chara{char} \opn\length{\ell} \opn\pd{pd} \opn\rk{rk}
\opn\projdim{proj\,dim} \opn\injdim{inj\,dim} \opn\rank{rank}
\opn\depth{depth} \opn\grade{grade} \opn\height{height}
\opn\embdim{emb\,dim} \opn\codim{codim} \opn\dim{dim}
\opn\sdepth{sdepth} \opn\sqdepth{sqdepth}

\opn\Tr{Tr} \opn\bigrank{big\,rank}
\opn\superheight{superheight}\opn\lcm{lcm}
\opn\trdeg{tr\,deg}%
\opn\reg{reg} \opn\lreg{lreg} \opn\ini{in}
\opn\div{div} \opn\Div{Div} \opn\cl{cl} \opn\Cl{Cl}
\opn\Spec{Spec} \opn\Supp{Supp} \opn\supp{supp} \opn\Sing{Sing}
\opn\Ass{Ass}  \opn\Min{Min}
\opn\Ann{Ann} \opn\Rad{Rad} \opn\Soc{Soc}
\opn\Ker{Ker} \opn\Coker{Coker} \opn\Am{Am} \opn\Hom{Hom}
\opn\Tor{Tor} \opn\Ext{Ext} \opn\End{End} \opn\Aut{Aut}
\opn\id{id}  \opn\e{e} \opn\hreg{hreg}

\opn\nat{nat} \opn\deg{deg} \opn\adeg{adeg} \opn\Hilb{Hilb}
\opn\pff{pf}%   \pf exists already
\opn\Pf{Pf} \opn\GL{GL} \opn\SL{SL} \opn\mod{mod} \opn\ord{ord}
\opn\aff{aff} \opn\con{conv} \opn\relint{relint} \opn\st{st}
\opn\lk{lk} \opn\cn{cn} \opn\core{core} \opn\vol{vol}
\opn\link{link}  \opn\infpt{infpt} \opn\mult{mult}
\opn\gr{gr} \opn\Std{Std}

\def\pot#1#2{#1[\kern-0.28ex[#2]\kern-0.28ex]}

\opn\dirlim{\underrightarrow{\lim}}
\opn\inivlim{\underleftarrow{\lim}}
\let\iso=\cong
\let\Union=\bigcup

\let\Dirsum=\bigoplus

\def\Implies{\ifmmode\Longrightarrow \else
     \unskip${}\Longrightarrow{}$\ignorespaces\fi}
\def\implies{\ifmmode\Rightarrow \else
     \unskip${}\Rightarrow{}$\ignorespaces\fi}
\def\iff{\ifmmode\Longleftrightarrow \else
     \unskip${}\Longleftrightarrow{}$\ignorespaces\fi}

\let\:=\colon
\newtheorem{Theorem}{Theorem}[section]
\newtheorem{Lemma}[Theorem]{Lemma}
\newtheorem{Corollary}[Theorem]{Corollary}
\newtheorem{Proposition}[Theorem]{Proposition}
\newtheorem{Remark}[Theorem]{Remark}

\newtheorem{Example}[Theorem]{Example}

\let\epsilon\varepsilon
\let\phi=\varphi
\let\kappa=\varkappa
\opn\dis{dis}
\def\pnt{{\raise0.5mm\hbox{\large\bf.}}}

\begin{document}

\title{Pretty clean monomial ideals and linear quotients}

\author{Ali Soleyman Jahan and Xinxian Zheng}
\address{Ali Soleyman Jahan, Fachbereich Mathematik und
Informatik, Universit\"at Duisburg-Essen, Campus Essen, 45117 Essen,
Germany} \email{ali.soleyman-jahan@stud.uni-duisburg-essen.de}

\address{Xinxian Zheng, Fachbereich Mathematik und
Informatik, Universit\"at Duisburg-Essen, Campus Essen, 45117
Essen, Germany} \email{xinxian.zheng@uni-essen.de}

\thanks{The second author is grateful for the financial support by DFG (Deutsche Forschungsgemeinschaft)
during the preparation of this work}


\maketitle

\begin{abstract}
We study basic properties of monomial ideals with linear quotients.
It is shown  that if the monomial ideal $I$ has linear quotients,
then the squarefree part of $I$ and each component of $I$ as well as
$\mm I$ have linear quotients, where $\mm$ is the graded maximal
ideal of the polynomial ring. As an analogy to the Rearrangement
Lemma of Bj\"orner and Wachs we also show that for a monomial  ideal
with linear quotients the admissible order of the generators can be
chosen degree increasingly.

As a generalization of the facet ideal of a forest, we define
monomial ideals of forest type and show that they are pretty clean.
This result recovers a recent result of Tuly and Villarreal about
the shellability of a clutter with the free vertex property. As
another consequence of this result we show that if $I$ is a monomial
ideal of forest type, then Stanley's conjecture on Stanley
decomposition holds for $S/I$. We also show that a clutter is
totally balanced if and only if it has the free vertex property.
\end{abstract}

\section*{Introduction}
Let $K$ be a field, $S=K[x_1,\ldots,x_n]$  the polynomial ring in
$n$ variables, and  $I\subset S$ a monomial ideal.  We  denote by
$G(I)$ the unique minimal monomial system of generators of  $I$. We
say that  $I$ has linear quotients, if there exists an order
$\sigma=u_1,\ldots,u_m$ of $G(I)$ such that  the ideal
$(u_1,\ldots,u_{i-1}):u_i$ is generated by a subset of the variables
for $i=2,\ldots, m$. We denote this subset by $q_{u_i,\sigma}(I)$.
Any order of the generators for which we have linear quotients will
be  called an admissible order.  Ideals with linear quotients were
introduced by Herzog and Takayama \cite{HT}. If each component of
$I$ has linear quotients, then we say $I$ has componentwise linear
quotients.

The concept of linear quotients, similarly as  the concept  of shellability,
is purely combinatorial. However both concepts have strong algebraic implications.
Indeed, an ideal with linear quotients has componentwise linear resolutions while
shellability of a simplicial complex implies that it is sequentially Cohen-Macaulay.
These similarities are not accidental. In fact, let $\Delta$ be a simplicial complex
and $I_\Delta$ its Stanley-Reisner ideal. It is well-known  that $I_\Delta$ has linear
quotients if and only if the Alexander dual of $\Delta$  is shellable. Thus at least
in the squarefree case  ``linear quotients" and ``shellability" are dual concepts.
On the other hand, linear quotients are not only defined for squarefree  monomial ideals,
and hence  this concept is more general than that of shellability.

In this paper we prove some fundamental properties of monomial
ideals with linear quotients. In general, the product of two ideals
with linear quotients need not to have linear quotients, even if one
of them is generated by a subset of the variables, see Example~
\ref{concaherzog}. However in Lemma \ref{product}, we show that if
$I\subset S$ is a monomial ideal with linear quotients, then $\mm I$
has linear quotients, where $\mm=(x_1,\ldots,x_n)$ is the graded
maximal  ideal of $S$.

Let $I$ be a monomial ideal with linear
quotients and $\sigma=u_1,\ldots,u_m$ an admissible order of $G(I)$.
It is not hard to see that $\deg u_i\geq \min \{\deg u_1,\ldots,\deg
u_{i-1}\}$, for all $i\in [m]=\{1,\ldots,m\}$. But this order need
not to be a degree increasing order. We show in Lemma~\ref{increasing},
that there exists a degree increasing admissible
order $\sigma'$ induced by $\sigma$. Furthermore, one has
$q_{u,\sigma}(I)= q_{u,\sigma'}(I)$ for any $u\in G(I)$, see
Proposition \ref{do not change}. This implies in particular the
 ``Rearrangement Lemma''  of Bj\"orner and Wachs \cite{BW}.

As a main result of Section 2, we show in Theorem \ref{linear}, that
any monomial ideal with linear quotients has componentwise linear
quotients, and hence it is componentwise linear. Conversely,
assuming that all components of $I$ have linear quotients, we can
prove that $I$ has linear quotients only under some extra
assumption, see Proposition \ref{noname}. It would be  of interest
to know whether  the converse of Theorem \ref{linear} is true in
general.

Herzog and Hibi showed in  \cite{HH} that a squarefree monomial
ideal $I$ is componentwise linear if and only if the squarefree part
of each component has a linear resolution. We would like to remark
that the ``only if'' part of this statement is true more generally.
Indeed for {\em any} componentwise linear monomial ideal, the
squarefree part of each component has a linear resolution.   Here we
prove a slightly different result by showing that if a monomial
ideal $I$ has  linear quotients, then the squarefree part of $I$ has
linear quotients. This together with Theorem \ref{linear} implies
that the squarefree part of each component of $I$ has again linear
quotients. As a corollary of the above facts we obtain that if
$\Delta$ is shellable, then each facet skeleton (see the definition
in Section 2) of $\Delta$ is shellable. Unless $\Delta$ is pure,
this result differs from the well-known fact that each skeleton of a
shellable simplicial complex is again shellable.

In Section 3, we give a large and combinatorially interesting class
$\mathcal I$ of monomial ideals which are pretty clean (Theorem
\ref{pretty clean}), and hence Stanley's conjecture on Stanley
decompositions \cite{St} holds for $S/I$. As another consequence of
Theorem \ref{pretty clean} we get the main result of \cite{Fa3},
which says that $S/I(\Delta)$ is sequentially Cohen-Macaulay for any
forest $\Delta$, as defined by Faridi \cite{Fa1}. The class
$\mathcal I$ is a non squarefree version of the class of  facet
ideals of forests. Any ideal in $\mathcal I$ is called a monomial
ideal of forest type. We show in Theorem \ref{freetree} that $I$ is
a monomial ideal of forest type if and only if $I$ has the free
variable property. Identifying a squarefree monomial ideal with a
clutter, Theorem~\ref{freetree} says that a clutter has the free
vertex property in the sense of Tuyl and Villarreal if and only if
the clutter corresponds to a forest in the sense of Faridi,
equivalently, a totally balanced clutter in the language of
hypergraphs. Let $\mathcal C$ be a clutter, and let
$\Delta_{\mathcal C}$ be  the simplicial complex whose
Stanley--Reisner ideal is the edge ideal of $\mathcal C$. In
\cite[Theorem 5.3]{TV} Villiarreal and Tuyl show that
$\Delta_{\mathcal C}$ is shellable if $\mathcal C$ has  the free
vertex property. Therefore Theorem~\ref{pretty clean} may be viewed
as a generalization of \cite[Theorem 5.3]{TV}.

In the last section we give some examples of quasi-forests. These
examples show that the facet ideal of a quasi-forest need not always
to be clean. It would be  interesting to classify all quasi-forests whose
facet ideals are clean.

\section{Preliminaries and background}
In this section we fix the terminology, review some notation on
simplicial complexes and setup some background.

A {\em simplicial complex} $\Delta$ over a set of vertices
$[n]=\{1,\ldots, n\}$ is a collection of subsets of $[n]$ with the
property that $i \in\Delta$ for all $i\in [n]$, and if $F\in\Delta$
then all the subsets of $F$ are also in $\Delta$ (including the
empty set). An element of $\Delta$ is called a {\em face} of
$\Delta$, and the maximal faces of $\Delta$ under inclusion are
called {\em facets}. We denote $\mathcal F(\Delta)$ the set of
facets of $\Delta$. The simplicial complex with facets $F_1,\ldots,
F_m$ is denoted by $\langle F_1,\ldots, F_m\rangle$. The {\em
dimension} of a face $F$ is defined as $|F|-1$, where $|F|$ is the
number of vertices of $F$. The dimension of the simplicial complex
$\Delta$ is the maximal dimension of its facets. A simplicial
complex $\Gamma$ is called a {\em subcomplex} of $\Delta$ if
${\mathcal F}(\Gamma)\subset {\mathcal F}(\Delta)$.

A subset $C$ of $[n]$ is called a {\em vertex cover} of $\Delta$, if
$C\cap F\neq \emptyset$ for all facets $F$ of $\Delta$. A vertex
cover $C$ is said to be minimal if no proper subset of $C$ is a
vertex cover of $\Delta$. Recently,  vertex cover algebra was
studied in \cite{HHT} and \cite{HHTZ}.

We denote by $S = K[x_1,\ldots,x_n]$  the polynomial ring in $n$
variables over a field $K$. To a given simplicial complex $\Delta$
on the vertex set $[n]$, the Stanley--Reisner ideal, whose
generators correspond to the non-faces of $\Delta$ is well studied,
see for example in \cite {St}, \cite {BH} and \cite {Hi} for
details. Another squarefree monomial ideal associated to $\Delta$,
so-called facet ideal, was first studied by Faridi \cite{Fa1}. The
ideal $I(\Delta)$ generated by all monomials $x_{i_1}\cdots x_{i_s}$
where $\{i_1,\ldots, i_s\}$ is a facet of  $\Delta$, is called the
{\em facet ideal of $\Delta$}. For a simplicial complex of dimension
1, the facet ideal is the {\em edge ideal}, which was first studied
by Villarreal \cite{Vi}.

The following definitions were first introduced by Faridi in \cite
{Fa1}. Let $\Delta$ be a simplicial complex. A facet $F$ of $\Delta$
is called a {\em leaf} if either $F$ is the only facet of $\Delta$,
or there exists a facet $G\neq F$ in $\Delta$ such that $F\cap
H\subseteq F\cap G$ for any facet $H\in\Delta$, $H\neq F$. The facet
$G$ is called a {\em branch} of $F$. A simplicial complex $\Delta$
is called a {\em tree} if it is connected and every nonempty
subcomplex of $\Delta$ has a leaf. A simplicial complex $\Delta$
with the property that every connected component is a tree is called
a {\em forest}. A vertex $t\in F$ is called a free vertex of $F$ if
$F\in \mathcal F(\Delta)$ is the unique facet which contains $t$. It
is easy to see that any leaf has a free vertex.

Recall that the {\em Alexander dual} $\Delta^\vee$ of a simplicial
complex $\Delta$ is the simplicial complex whose faces are
$\{[n]\setminus F\:\; F\not\in \Delta\}$.  Let $I$ be a squarefree
monomial ideal in $S$. We denote by $I^\vee$ the squarefree
 monomial ideal which minimally generated by all monomials $x_{i_1}\cdots x_{i_k}$, where
$(x_{i_1},\ldots, x_{i_k})$ is a minimal prime ideal of $I$. It is
easy to see that for any simplicial complex $\Delta$, one has
$I_{\Delta^\vee}=(I_\Delta)^\vee$. Let $\Delta^c=\langle
[n]\setminus F\: F\in\mathcal F(\Delta)\rangle$. Then
$I_{\Delta^\vee}=I(\Delta^c)$, see \cite{HHZ}.

 For any set $U\subset [n]$, we denote $u=\prod_{j\in U} x_j$ the
squarefree monomial in $S$ whose support is $U$. In general, for
any monomial $u\in S$, the {\em support} of $u$ is $\supp
(u)=\{j\: x_j\mid u\}$.

\begin{Remark}
\label{cover} {\em Let $\Delta$ be a simplicial complex on $[n]$.
Then $$G(I(\Delta)^\vee)=\{u=\prod_{j\in U} x_j\:\; \text{where $U$
is a minimal vertex cover of } \Delta\}.$$}
\end{Remark}

Now we recall the definition of clean and pretty clean modules of
the type $S/I$, where $I\subset S$ is a monomial ideal. According
to \cite{HP}, a  filtration $\mathcal F\:I=I_0\subset I_1\subset
\cdots\subset I_r=S$ of $S/I$ is called a {\em pretty clean
filtration } if
\begin{enumerate}
\item[(a)] for all $j$ one has $I_j/I_{j-1}\iso S/P_j$ where $P_j$
is a monomial prime ideal;

\item[(b)] for all $i<j$, if $P_i\subset P_j$,
then $P_i=P_j$.
\end{enumerate}
The set of prime ideals $\{P_1,\ldots,P_r\}$ is called the {\em
support} of $\mathcal F$ and denoted by $\Supp(\mathcal F)$. The
module $S/I$ is called {\em pretty clean} if it has a pretty
 clean filtration.

Dress \cite{Dr} calls the ring $S/I$  {\em clean}, if there exists a
chain of ideals as above such that all the $P_i$ are minimal prime
ideals of $I$. By an abuse of notation we call $I$ (pretty) clean if
$S/I$ is (pretty) clean. Obviously, any clean ideal is pretty clean.
If $I$ is a squarefree monomial ideal, then pretty clean implies
also clean. The following fact was first shown by Dress.

\begin{Theorem} \cite{Dr}
\label{Dress} Let $\Delta$ be a simplicial complex  and
$I=I_{\Delta}\subset S$ its Stanley-Reisner ideal. Then the
simplicial complex $\Delta$ is (non-pure) shellable  if and only
if $I_{\Delta}$ is clean.
\end{Theorem}

 The following notion is  important for our later discussion.
 Let $I=(u_1,\ldots,u_m)$ be a monomial ideal in $S$. According to
\cite {HT}, the monomial ideal $I$ has linear quotients if one can
order the set of minimal generators of $I$,
$G(I)=\{u_1,\ldots,u_m\}$, such that the ideal
$(u_1,\ldots,u_{i-1}):u_i$ is generated by a subset of the variables
for $i=2,\ldots,m$.  This means for each $j<i$, there exists a $k<i$
such that $u_k:u_i=x_t$ and $x_t\mid u_j:u_i$, where $t\in [n]$ and
$u_k:u_i=u_k/\gcd(u_k,u_i)$. In the case that $I$ is squarefree, it
is enough to show that for each $j<i$, there exists a $k<i$ such
that $u_k:u_i=x_t$ and $x_t\mid u_j$. Such an order  of generators
is called an {\em admissible order} of $G(I)$. Let $\sigma =
u_1,\ldots,u_m$ be an admissible order of $G(I)$. We denote by
$q_{u_j,\sigma}(I)\subset \{x_1,\ldots,x_n\}$ the set of minimal
generators of $(u_1,\ldots,u_{j-1}):u_j$.

It is known that if $I$ is a monomial ideal with linear quotients
and generated in one degree, then $I$ has a linear resolution. See
for example in \cite {Zh} an easy proof.

\begin{Remark}
\label{change} {\em For an ideal which has linear quotients, there
might exist several admissible orders. For example, let
$I=(x_1x_2,x_1x_3^2x_4,x_2x_4)\subset K[x_1,x_2,x_3,x_4]$. Then
$\sigma_1=x_1x_2,x_1x_3^2x_4,x_2x_4$ and
$\sigma_2=x_1x_2,x_2x_4,x_1x_3^2x_4$ both are admissible orders of
$G(I)$.}
\end{Remark}

The following result relates squarefree monomial ideals with linear
quotients to (non-pure) shellable simplicial complexes. The concept
non-pure shellability was first defined by Bj\"orner and Wachs
\cite[Definition 2.1]{BW}.

\begin{Theorem}\cite[Theorem 1.4]{HHZ}
\label{hhz} Let $\Delta$ be a simplicial complex and
$(I_{\Delta})^\vee$ the Alexander dual of its Stanley-Reisner
ideal. Then $\Delta$ is (non-pure) shellable  if and only if
$(I_{\Delta})^\vee$ has linear quotients.
\end{Theorem}

Combining Theorem \ref{Dress} and Theorem \ref{hhz}, we get the
following

\begin{Corollary}
\label{important} Let $I\subset S$ be a squarefree  monomial ideal.
Then $I$ is clean if and only if $I^\vee$ has linear quotients.
\end{Corollary}

\section{Monomial ideals with linear quotients}

In this section we prove some fundamental properties of ideals with linear quotients.

Let $I\subset S$ be a monomial ideal with linear quotients and
$u_1,\ldots,u_m$ an admissible  order  of $G(I)$. It is easy to
see that $\deg u_i\geq \min\{\deg u_1,\ldots, \deg u_{i-1}\}$ for
$i=2,\ldots, m$. In particular,  $\deg u_1=\min\{\deg u_1,\ldots,
\deg u_{m}\}$. But in general, this order need not to be a degree
increasing order. For example, the ideal $I=(x_1x_2, x_1x_3^2x_4,
x_2x_4)$ has linear quotients in the given order, but $\deg
x_1x_3^2x_4> \deg x_2x_4$.

In the following lemma we show that for any ideal with linear
quotients there exists an admissible  order $u_1,\ldots, u_m$ of
$G(I)$ such that $\deg u_i\leq \deg u_{i+1}$ for $i=1,\ldots, m-1$.
We call such an order a {\em degree increasing admissible  order}.

\begin{Lemma}
\label{increasing} Let $I\subset S$ be a monomial ideal with linear
quotients. Then there is a degree increasing admissible order of
$G(I)$.
\end{Lemma}

\begin{proof}
We use induction on $m$, the number of generators of $I$, to prove
the statement. If $m=1$, there is nothing to show.

Assume $m>1$ and $u_1,\ldots, u_m$ is an admissible  order. It is
clear that $J=(u_1,\ldots, u_{m-1})$ has linear quotients with the
given order. By induction hypothesis, we may assume that $\deg
u_i\leq \deg u_{i+1}$ for $i=1,\ldots, m-2$.  Assume that $\deg
u_{m-1}> \deg u_m$. Let $j+1$ be the smallest integer such that
$\deg u_{j+1}> \deg u_{m}$. By the observation before this lemma,
one sees that $j+1\neq 1$. Now we show that
$u_1,\ldots,u_j,u_m,u_{j+1}, \ldots,u_{m-1}$ is an admissible  order
which is obviously degree increasing.

We need to prove that $(u_1,\ldots,u_j):u_m$ and
$(u_1,\ldots,u_j,u_m,u_{j+1}, u_{p-1}):u_p$ are generated in degree
one, for $p=j+1,\ldots, m-1$. Since $\deg u_m <\deg u_q$ for
$q=j+1,\ldots, m-1$, we have $\deg (u_q:u_m) > 1$. Since
$u_1,\ldots, u_m$ is an admissible  order, for any $r\leq j$, there
exists a $k\leq j$ such that $\deg (u_k:u_m)=1$ and $u_k:u_m \mid
u_r:u_m$. This shows that $(u_1,\ldots,u_j):u_m$ is generated in
degree one. Now let $j+1\leq p\leq m-1$. It is clear that for any
$r\leq p-1$, there exists a $k\leq p-1$ such that $\deg (u_k:u_p)=1$
and $u_k:u_p\mid u_r:u_p$, since the ideal $(u_1,\ldots,u_j,u_{j+1},
\ldots,u_p)$ has linear quotients in this order. It remains to show
that there is an $h<p$ such that $\deg (u_h:u_p)=1$ and $u_h:u_p\mid
u_m:u_p$. Since $u_1,\ldots,u_j,u_{j+1},\ldots, u_m$ is an
admissible order and $\deg u_m<\deg u_q$ for $q=j+1,\ldots, m-1$,
there exists a $k\leq j$ such that $u_k:u_m=x_d$ and $x_d\mid
u_p:u_m$ for some $d\in [n]$. Since $u_1,\ldots,u_j,u_{j+1},
\ldots,u_p$ is an admissible  order, there exists an $h<p$ such that
$u_h:u_p=x_b$ and $x_b\mid u_k:u_p$ for some $b\in [n]$.

We claim  that $x_b\mid u_m:u_p$. In order to prove this we first
show that $b\neq d$.
 Suppose $b=d$. Then we have
$x_d=u_k:u_m$ and $x_d=x_b\mid u_k:u_p$. Hence $\deg_{x_d}
u_k=\deg_{x_d} u_m+1$ and $\deg_{x_d} u_k\geq \deg_{x_d} u_p+1$,
where by $\deg_{x_d} u$ we mean the degree of $x_d$ in $u$.
Therefore $\deg_{x_d} u_m\geq \deg_{x _d}u_p$, which is a
contradiction, since $x_d\mid u_p:u_m$.

Now  since $x_b=u_h:u_p$ and $x_b\mid u_k:u_p$, we have $\deg_{x_b}
u_h=\deg_{x_b} u_p+1$ and $\deg_{x_b} u_k\geq \deg_{x_b} u_p+1$. On
the other hand, since $x_d=u_k:u_m$ and $b\neq d$, we have
$\deg_{x_b} u_m\geq \deg_{x_b} u_k\geq \deg_{x_b} u_p+1 > \deg_{x_b}
u_p$. This implies that $x_b\mid u_m:u_p$.
\end{proof}

If $\sigma =u_1,\ldots, u_m$ is any admissible order of $G(I)$, we
denote by $\sigma'=u_{i_1},\ldots, u_{i_m}$  the degree increasing
admissible order derived from $\sigma$ as given in Lemma
\ref{increasing}. The order $\sigma'$ is called the degree
increasing admissible order induced by $\sigma$. Attached to an
admissible order $\sigma$ are the sets $q_{u,\sigma}(I)$ as defined
in the previous section. We have the following result.

\begin{Proposition}
\label{do not change} Let $I$ be a monomial ideal with linear
quotients with respect to the admissible order $\sigma$ of the generators.
Then for all $u\in G(I)$ we have
\[
q_{u,\sigma}(I)= q_{u,\sigma'}(I).
\]
\end{Proposition}

\begin{proof}
Let $\sigma=u_1,\ldots,u_m$ and $\sigma'=u_{i_1},\ldots,u_{i_m}$.
Suppose $u=u_k$ in  $\sigma$ and $u=u_{i_t}$
 in $\sigma'$. Let $x_d\in q_{u,\sigma}(I)$, for some $d\in [n]$,
 then there exists $j<k$ such that $u_j:u_k=x_d$. In particular,
 $\deg u_j\leq \deg u_k$. According to the definition of $\sigma'$,
 $u_j$ comes before $u_{i_t}$ and hence $x_d\in q_{u,\sigma'}(I)$.

Conversely,  let $x_d\in q_{u,\sigma'}(I)$ for some $d\in [n]$. Then
there exists an $i_j$ with $j<t$, such that $u_{i_j}:u_{i_t}=x_d$.
We may assume that $j$ is the smallest integer with this property
and $u_{i_j}=u_r$ in $\sigma$.

Suppose $x_d\not\in q_{u,\sigma}(I)$. Then $r>k$ and $\deg u_r<\deg
u_k$ according to the definition of $\sigma'$. Therefore $u_r=x_du$
and $u_k=wu$ where $u$ and $w$ are monomials with $\deg w\geq 2$ and
$x_d\nmid w$. Since $u_1,\ldots,u_r$ is an admissible order and
$k<r$, there exists an $s<r$ such that $u_s:u_r=x_b$ and $x_b\mid
u_k:u_r=w$ ($b\neq d$). Hence $\deg u_s\leq \deg u_r=\deg u_{i_j}$.
Therefore $u_s=u_{i_l}$ with $l<j$.

It follows that $\deg_{x_b} u_s=\deg_{x_b} u_r+1\leq \deg_{x_b}
u_k$, $\deg_{x_c} u_s\leq \deg_{x_c} u_r\leq \deg_{x_c} u_k$ for any
$c\neq d,b$, and $\deg_{x_d} u_s\leq\deg_{x_d} u_r=\deg_{x_d}
u_k+1$. If $\deg_{x_d} u_s<\deg_{x_d} u_k+1$, then we have $u_s\mid
u_k$, a contradiction. Therefore $\deg_{x_d} u_s=\deg_{x_d} u_k+1$,
and hence $x_d=u_s:u_k=u_{i_l}:u_{i_t}$, contradicting the choice of
$j$.
\end{proof}

 Let $\Delta$ be a simplicial complex with $\mathcal F(\Delta)=\{F_1,\ldots,F_m\}$.  Then
 $I_\Delta=\bigcap_{i=1}^m P_{F_i}$
where $P_{F_i}=(x_j\: j\not\in F_i)$, see \cite[Theorem 5.4.1]{BH}.
It follows from \cite[Lemma 1.2]{HHZ} that
$I_{\Delta^\vee}=(u_1,\ldots,u_m)$, where $u_i=\prod_{j\not\in F_i}
x_j$. We follow the notation in \cite{BW}: if
$\delta=F_1,\ldots,F_m$ is any order of facets of $\Delta$, then we
set $\Delta_k=\langle F_1,\ldots,F_k\rangle$ and
$R_\delta(F_k)=\{i\in F_k\: F_k-\{i\}\in \Delta_{k-1}\}$ for any
$k\in[m]$.

We observe the following simple but important fact: $\Delta$ is
shellable with shelling $\delta=F_1,\ldots, F_m$ if and only if
$I_{\Delta^\vee}$ has linear quotients with the admissible order
$\sigma=u_1,\ldots,u_m$. Moreover, if the equivalent conditions
hold, then $R_\delta(F_k)=q_{u_k,\sigma}(I_{\Delta^\vee}).$

\medskip
As an immediate consequence of Lemma \ref{increasing}, Proposition
\ref{do not change} and the observation  above we rediscover the
following well-known ``Rearrangement Lemma" of Bj\"orner and Wachs
\cite[Lemma 2.6]{BW}.
\begin{Corollary}
\label{shelling} Let $\delta=F_1,\ldots, F_m$  be a shelling of the
simplicial complex $\Delta$.
 There exists a shelling $\delta'=F_{i_1},\ldots, F_{i_m}$ of
$\Delta$ induced by $\delta$ such that $\dim F_{i_k}\geq \dim
F_{i_{k+1}}$ for $k=1,\ldots, m-1$. Furthermore we have
$R_\delta(F)=R_{\delta'}(F)$ for any facet $F$ of $\Delta$.
\end{Corollary}

It is known that the product of two ideals with linear quotients
need not to have again linear quotients, even if one of them is
generated by linear forms. Such an example was given by Conca and
Herzog \cite{CH}.

\begin{Example}
\label{concaherzog} {\em Let $R=k[a,b,c,d]$, $I=(b,c)$ and
$J=(a^2b,abc,bcd,cd^2)$. Then $J$ has linear quotients, and $I$ is
generated by a subset of the variables. But the product $IJ$ has no
linear quotients (not even a linear resolution). }
\end{Example}

However, we have the following

\begin{Lemma}
\label{product} Let $I\subset S$ be a monomial ideal. If $I$ has
linear quotients, then $\mm I$ has linear quotients, where
$\mm=(x_1,\ldots,x_n)$ is the graded maximal ideal of $S$.
\end{Lemma}

\begin{proof}
We may assume $G(I)=\{u_1,\ldots, u_m\}$ and $u_1,\ldots, u_m$ is a degree increasing
admissible order. We prove the assertion by using induction on $m$.

The case $m=1$  is trivial. Let $m>1$. Consider the multi-set
$$T=\{u_1x_1,\ldots,u_1x_n,u_2x_1,\ldots,u_2x_n,\ldots,
u_mx_1,\ldots, u_mx_n\}.$$ It is a  system of generator of $\mm
I$. If $u_ix_j\mid u_rx_s$ for some $i<r$, then we remove $u_rx_s$
from $T$. In this way, we get the minimal set
 $$T'=\{u_ix_j\:i=1,\ldots,m, j\in A_i\}$$
of monomial generators of $\mm I$, where $A_1=[n]$ and
$A_i\subset[n]$ for $i=2,\ldots, m$. We shall order $G(\mm I)$ in
the following way:  $u_kx_l$ comes before $u_tx_s$ if $k<t$ or $k=t$
and $l<s$. Now we show that the above order $\sigma$ of $G(\mm I)$
is an admissible  order. We define the order of the generators of
$\mm (u_1\ldots,u_{m-1})$  in the same way as we did for $\mm I$.
Then the ordered sequence $\tau$ of the generators of $\mm
(u_1\ldots,u_{m-1})$ is an initial sequence of $\sigma$. Moreover,
by induction hypothesis, $\tau$ is an admissible order of $G(\mm
(u_1\ldots,u_{m-1}))$.

For a given  $j\in A_m$  let $J$ be the ideal generated by all monomials
in $T'$ which come before $u_mx_j$ with respect to $\sigma$. It remains to be shown that $J:u_mx_j$
is generated by monomials of degree 1.

Let $u_kx_l\in G(J)$. If $k=m$, then
$u_kx_{l}:u_mx_{j}=x_{l}$.
If $k<m$, then we shall find an element  $u_rx_s\in G(J)$ and $t\in [n]$ such that
$u_rx_{s}:u_mx_{j}=x_t$ and $x_t\mid u_kx_{l}:u_mx_{j}$. Indeed since
$u_1,\ldots, u_m$ is an admissible  order of $G(I)$, there exists $q<m$ such
that $u_q:u_m=x_t$  and $x_t\mid u_k:u_m$. This implies that
$u_qx_{j}:u_mx_{j}=u_q:u_m=x_1$. Since $u_qx_{j}\in \mm
I$,  there exists, by the definition of $\sigma$, a monomial  $u_rx_{s}\in G(J)$  such that
$u_rx_{s}\mid u_qx_{j}$.

We claim that  $u_rx_s:u_mx_{j}=x_t$ and $x_t\mid u_kx_{l}:u_mx_{j}$.
Notice that $u_rx_s:u_mx_{j}\mid u_qx_{j}:u_mx_{j}=x_t$. If
$u_rx_s:u_mx_{j}\neq x_t$, then $u_rx_s:u_mx_{j}=1$, that is,
$u_rx_s\mid u_mx_j$ which contradicts the fact that $j\in A_m$. This shows that  $u_rx_s:u_mx_{j}=x_t$.

 Since $x_t\mid u_k:u_m$, it is enough to
show that $x_t\neq x_{j}$ in order to prove that $x_t\mid
u_kx_{l}:u_mx_{j}$. Assume that $x_t=x_{j}$. Since $u_q:u_m=x_t$, we
have $u_q=x_tu$ for some monomial $u$ such that $u\mid u_m$. Since
$\deg u_q\leq\deg u_m$, it follows that  $u_m=uw$ for some monomial
$w$ with $\deg w\geq 1$ and $x_t\nmid w$. Hence there exists some
variable $x_d$ with $d\neq t$ such that $x_d\mid w$. But then
$x_du_q=x_dux_t\mid wux_t=u_mx_{j}$, contradicting again the fact
that $j\in A_m$.
\end{proof}

\begin{Remark}
{\em The converse of the above lemma is not true. For example, let
$I=(ab,cd)\subset K[a,b,c,d]$. Then $\mm
I=(a^2b,ab^2,abc,abd,acd,bcd,c^2d,cd^2)$ has linear quotients in the
given order, but $I$ has no linear quotients.}
\end{Remark}

  Now we present the main theorem of this section.

\begin{Theorem}
\label{linear} Let $I\subset S$ be a monomial ideal. If $I$ has
linear quotients, then $I$ has componentwise linear quotients.
\end{Theorem}

\begin{proof}
By Lemma \ref{product} and Lemma \ref{increasing}, we may assume
that $I$ is generated by monomials of two different degrees $a$ and
$a+1$. We denote by $I_{\langle a\rangle}$  the ideal generated by the
$a$-th graded component of the ideal $I$. Let $G(I)=\{u_1,\ldots,u_s,v_1,\ldots,v_t\}$,
where $\deg u_i=a$ for $i=1,\ldots, s$ and $\deg v_j=a+1$ for
$j=1,\ldots, t$. By Lemma \ref{increasing}, we may assume that
$u_1,\ldots,u_s,v_1,\ldots,v_t$ is an admissible order, hence $I_a$
has linear quotients. Now we show that $I_{\langle a+1\rangle}$ has also linear
quotients.

We have $I_{\langle a+1\rangle}=\mm (u_1,\ldots,u_s) + (v_1,\ldots,v_t)$. Let
$G(I_{\langle a+1\rangle})=\{w_1,\ldots,w_l,v_1,\ldots,v_t\}$, where $w_1,\ldots,w_l$ is
ordered as in Lemma \ref{product}.  In particular, $w_1,\ldots,w_l$ is an admissible order. We only need
to show that $(w_1,\ldots,w_l,v_1,\ldots,v_{p-1}):v_p$ is generated
by a subset of the variables, for $1\leq p\leq t$.

First  we consider $v_j:v_p$ where $j<p$. Since
$u_1,\ldots,u_s,v_1,\ldots,v_t$ is an admissible order of $G(I)$,
there exists some $u\in\{u_1,\ldots,u_s,v_1,\ldots,v_t\}$ and $d\in
[n]$ such that $u:v_{p}=x_d$ and $x_d\mid v_j:v_{p}$. If
$u\in\{v_1,\ldots,v_t\}$ we are done. So we may assume
$u\in\{u_1,\ldots,u_s\}$. Therefore,  $\deg u=\deg v_{p}-1$. Since
$u:v_{p}=x_d$, $\deg_{x_d} u=\deg_{x_d} v_{p}+1$ and $\deg_{x_b}
u\leq\deg_{x_b} v_p$ for any $b\neq d$. Since $\deg u<\deg v_p$,
there exists a variable $x_c$ with $c\neq d$ such that $\deg_{x_c} u
\leq\deg_{x_c} v_p-1$. Since  $x_cu\in\mm I_{\langle a\rangle}$, one
has $x_cu=w_k$ for some $k\leq l$. All this  implies that
$\deg_{x_d}w_k=\deg_{x_d}u=\deg_{x_d} v_p+1$ and $\deg_{x_b}
w_k\leq\deg_{x_b} v_p$ for any $b\neq d$. Therefore $w_k:v_p=x_d$
and $x_d\mid v_j:v_p$.

It remains to consider $w_j: v_p$. In this case $w_j=x_bu_{i}$ for some
$i \in [s]$  and some $b\in [n]$. Since
$u_1,\ldots,u_s,v_1,\ldots,v_t$ is an admissible order, there exists
some $u\in\{u_1,\ldots,u_s,v_1,\ldots,v_t\}$ and $d\in [n]$ such that
$u:v_p=x_d$ and $x_d\mid u_{i}:v_p$.
Therefore $x_d\mid w_j:v_p$, since $u_{i}:v_p\mid
w_j:v_p$. If $u\in\{v_1,\ldots,v_t\}$, then we are done. So we
may assume $u\in\{u_1,\ldots,u_s\}$. Then, as before, there
exists a variable $x_c$ with $c\neq d$ such that $x_cu\in\mm I_{\langle a\rangle}$,
$\deg_{x_d} x_cu=\deg_{x_d} u=\deg_{x_d}v_p+1$ and $\deg_{x_b}
x_cu\leq\deg_{x_b} v_p$ for any $b\neq d$. This implies that
$x_cu:v_p=x_d$ and $x_d\mid w_j:v_p$.
\end{proof}

\begin{Corollary}
\label{complinear} If $I\subset S$ is a monomial ideal with linear
quotients, then $I$ is componentwise linear.
\end{Corollary}

We do not know if the converse of Theorem \ref{linear} is true in
general. However we could prove the following:

\begin{Proposition}
\label{noname} Let $I$ be a monomial ideal with componentwise linear
quotients. Suppose for each component $I_{\langle a\rangle}$ there
exists an admissible order $\sigma_a$ of $G(I_{\langle a\rangle})$
with the property that the elements of $G(\mm I_{\langle
a-1\rangle})$ form the initial part of $\sigma_a$. Then $I$ has
linear quotients.
\end{Proposition}

\begin{proof} We chose the order $\sigma  =u_1,\ldots,u_s $ of $G(I)$  such that  that $i<j$ if
$\deg u_i<\deg u_{j}$ or $\deg u_i=\deg u_{j}=a$ and $u_i$ comes before $u_j$ in  $\sigma_a$.

We show that $(u_1,\ldots,u_{p-1}):u_p$ is generated by linear
forms. If $\deg u_1=\deg u_p$, then there is nothing to prove.

Now assume that $\deg u_1 <\deg u_p=b$. Let $l<p$ be the largest number such that
$\deg u_l< b$.  Then, by our assumption,  there exists an admissible order
$w_1,\ldots, w_t, u_{l+1},\ldots,u_p$ where $w_1,\ldots, w_t\in G(\mm I_{\langle b-1\rangle})$.

Let $j<p$ and suppose that  $\deg(u_j:u_p)\geq 2$. Let $m$ be a
monomial such that $\deg (mu_j)=\deg u_p$  and $mu_j:u_p=u_j:u_p$.
Since $mu_j\in\{w_1,\ldots, w_t, u_{l+1},\ldots,u_{p-1}\}$  there exists $w\in
\{w_1,\ldots, w_t, u_{l+1},\ldots,u_{p-1}\}$ and some $d\in [n]$ such that
$w:u_p=x_d$ and $x_d\mid u_j:u_p$ because $mu_j:u_p =u_j:u_p$.

If $w\in \{u_{l+1},\ldots,u_{p-1}\}$, then we are done. On the other
hand, if $w\in \{w_1,\ldots, w_t\}$, then $w=m'u_i$ for some $i\leq
l$ and some monomial $m'$. Since $w:u_p=x_d$, one has $\deg_{x_b}
w\leq\deg_{x_b} u_{p}$ for all $b\neq d$. Hence $x_d$ does not
divide  $m'$, otherwise $u_i\mid u_p$ which contradicts the fact
that $u_i, u_p\in G(I)$. Therefore $x_d=u_i:u_p$ and $x_d\mid
u_j:u_p$.
\end{proof}

Let $I\subset S$ be a monomial ideal. We denote by $I^*$ the
monomial ideal generated by the squarefree monomials in $I$ and call
it the squarefree part of $I$.  Indeed $I^*=(u\: u\in G(I) \text{
and $u$ is squarefree})$. We follow \cite{HH} and denote by
$I_{[a]}$ the squarefree part of $I_{\langle a\rangle}$. In
\cite[Proposition 1.5]{HH}, the authors proved that if $I$ is
squarefree, then $I_{\langle a\rangle}$ has a linear resolution if
and only if $I_{[a]}$ has a linear resolution. Indeed for the ``only
if'' part one does not need the assumption that $I$ is squarefree.
We have the following slightly different result.

\begin{Proposition}
\label{sfpart} Let $I$ be a monomial ideal in $S$. If $I$ has linear
quotients, then $I^*$ has linear quotients.
\end{Proposition}

\begin{proof}
Let $u_1,\ldots,u_m$ be an admissible order of $G(I)$. Assume
$I^*=(u_{i_1},\ldots,u_{i_t})$, where $1\leq i_1<i_2<\cdots <i_t\leq
m$. We shall show $u_{i_1},\ldots,u_{i_t}$ is an admissible order of
$G(I^*)$ by using induction on $m$.

The case $m=1$ is trivial. Now assume $m>1$. It is clear that
$(u_{i_1},\ldots,u_{i_{t-1}})$ is the squarefree part of the
monomial ideal $(u_1,\ldots,u_{i_{t-1}})$, where
$u_1,\ldots,u_{i_{t-1}}$ is an admissible order. By induction
hypothesis $u_{i_1},\ldots,u_{i_{t-1}}$ is an admissible order of
$G((u_{i_1},\ldots,u_{i_{t-1}}))$. Consider $u_{i_j}:u_{i_t}$ with
$j<t$. Since $u_1,\ldots,u_m$ is an admissible order of $G(I)$,
there exists $k<i_t$ and some $d\in [n]$ such that $u_k:u_{i_t}=x_d$
and $x_d\mid u_{i_j}:u_{i_t}$.  Since $u_{i_j}$ and $u_{i_t}$ are
squarefree, we have $x_d\nmid u_{i_t}$. On the other hand, since
$u_k:u_{i_t}=x_d$, one has $\deg_{x_d} u_k=1$ and $\deg_{x_b}
u_k\leq \deg_{x_b} u_{i_t}\leq 1$ for any  $b\neq d$. Hence $u_k\in
G(I^*)$.
\end{proof}

Combining Proposition \ref{sfpart} with Theorem \ref{linear}, we
obtain:
\begin{Corollary}
\label{sqfree} Let $I\subset S$ be a monomial ideal with linear
quotients. Then $I_{[a]}$ has linear quotients for all $a$.
\end{Corollary}

\begin{Remark}
\label{exterior} {\em All results concerning linear quotients proved
in this section are correspondingly valid for monomial ideals in the
exterior algebra.}
\end{Remark}

Let $\Delta$ be a $d$-dimensional simplicial complex. We define the
{\em $1$-facet skeleton} of $\Delta$ to be the simplicial complex
$$\Delta^{[1]}=\langle G\: G\subset F\in \mathcal F(\Delta) \text{ and }
|G|=|F|-1\rangle.$$  Recursively, the {\em $i$-facet skeleton} is
defined to be the 1-facet skeleton of $\Delta^{[i-1]}$, for
$i=1,\ldots,d$. For example
 if $\Delta=\langle \{1,2,3\},\{2,3,4\},\{4,5\}\rangle$, then
 $$\Delta^{[1]}=\langle
 \{1,2\},\{1,3\},\{2,3\},\{2,4\},\{3,4\},\{5\}\rangle \text{ and } \Delta^{[2]}=\langle
 \{1\},\{2\},\{3\},\{4\}\rangle.$$
 If $\Delta$ is pure of dimension $d$, then the $i$-facet skeleton of
$\Delta$ is just the $(d-i)$-skeleton of $\Delta$. Now let $\Gamma$
be a shellable simplicial complex with facets $F_1\ldots,F_m$. It is
known that any skeleton of $\Gamma$ is shellable, see \cite[Theorem
2.9]{BW}. Since $I_\Gamma=\bigcap_{i=1}^m P_{F_i}$ where
$P_{F_i}=(x_j\: j\not\in F_i)$, we have
$(I_{\Gamma})^\vee=(u_1,\ldots,u_m)$, where $u_i=\prod_{j\not\in
F_i} x_j$. By Theorem \ref{hhz} $(I_\Gamma)^\vee$ has linear
quotients. Hence $\mm (I_\Gamma)^\vee$ and the squarefree part of
$\mm (I_\Gamma)^\vee$ have linear quotients by Lemma \ref{product}
and Proposition \ref{sfpart}. It is not hard to see that the
squarefree part of $\mm (I_\Gamma)^\vee$ is the Alexander dual of
$I_{\Gamma^{[1]}}$. Hence our discussions yield the following:

\begin{Corollary} If $\Gamma$ is a shellable simplicial complex of dimension $d$,
 then $\Gamma^{[i]}$ is shellable, for $i\leq d$. In particular, if
 $\Gamma$ is pure, then any skeleton of $\Gamma$ is again shellable.
\end{Corollary}

\section{A class of pretty clean monomial ideals}
In this section we study a class of monomial ideals which are pretty
clean. This class is a generalization of the class of facet ideals
of forests.

Let $I\subset S$ be a squarefree monomial ideal. There is a unique
simplicial complex $\Delta$ such that $I=I(\Delta)$. Now we
generalize the concept of the  facet ideal of a forest as follows:
Let $I$ be a monomial ideal (not necessarily squarefree) with
$G(I)=\{u_1,\ldots,u_m\}$. A variable $x_i$ is called a free
variable of $I$ if there exists a $t\in[m]$ such that $x_i\mid u_t$
and $x_i\nmid u_j$ for any $j\neq t$.  A monomial $u_t$ is called a
{\em leaf} of $G(I)$ if $u_t$ is the only element in $G(I)$ or there
exists a $j\in [m]$, $j\neq t$ such that $\gcd(u_t,u_i)\mid
\gcd(u_t,u_j)$ for all $i\neq t$. In this case $u_j$ is called a
{\em branch} of $u_t$. We say that $I$ is a {\em monomial ideal of
forest type} if any subset of $G(I)$ has a leaf. It is clear that
any monomial ideal of forest type has a free variable.

Let
$(X_1,X_2)=(\{x_{i_1},\ldots,x_{i_r}\},\{x_{j_1},\ldots,x_{j_s}\})$,
where $X_1$, $X_2$ are subsets of $X=\{x_1,\ldots,x_n\}$ and
$X_1\cap X_2=\emptyset$. Let $I$ be a monomial ideal in
$S=K[x_1,\ldots,x_n]$. As in \cite{TV} we define the {\em minor} of
$I$ with respect to $(X_1,X_2)$ to be the ideal
$I_{(X_1,X_2)}\subset K[X\setminus (X_1\cup X_2)]$ obtained from $I$
by setting $x_{i_k}=0$ for $k=1,\ldots,r$ and $x_{j_l}=1$ for
$l=1,\ldots,s$. In particular, $I_{(\emptyset,\emptyset)}=I$. One
says that the ideal $I$ has the {\em free variable property} if all
minors of $I$ have free variables. The following lemma is a
generalization of \cite[Lemma 4.5]{Fa2} to any monomial ideal of
forest type.

\begin{Lemma}
\label{prime} Let $I$  be a monomial ideal of forest type and
$X'=\{x_{j_1},\ldots,x_{j_s}\}$ a subset of $X$. Then
$I_{(\emptyset,X')}$ is again a monomial ideal of forest type.
\end{Lemma}

\begin{proof}
We only need to prove that $I_{(\emptyset,\{x_{j_1}\})}$ is a
monomial ideal of forest type. Hence we may assume that
$X'=\{x_{i}\}$. Let $G(I)=\{u_1,\ldots,u_m\}$. We write $u_j=\bar
u_jx_i^{a_j}$, where $a_j\geq 0$ and $x_i\nmid \bar u_j$ for
$j=1,\ldots,m$. Let $A$ be any subset of $G(I_{(\emptyset,X')})$.
Consider the subset $A'=\{u_j\: \bar u_j\in A\}$ of $G(I)$. Since
$I$ is a monomial ideal of forest type, $A'$ has a leaf $u_p$. This
means that there exists a $u_k\in A'$ such that $\gcd(u_p,u_q)\mid
\gcd(u_p,u_k)$ for all $u_q\in A'$ with $q\neq p$.

Let $\gcd (u_p,u_q)=v_qx_i^{a_q}$ and $\gcd(u_p,u_k)=v_kx_i^{a_k}$,
where $v_q$, $v_k$ are monomials and $x_i\nmid v_q$, $x_i\nmid v_k$.
Then $\gcd (\bar u_p,\bar u_q)=v_q$ which divides $\gcd(\bar
u_p,\bar u_k)=v_k$ for all $\bar u_q\in A$ with $q\neq p$. Hence
$\bar u_p$ is a leaf of $A$.
\end{proof}

Now we recall the following fact from \cite{K}, which is needed for
the proof of the next proposition.

\begin{Lemma} \label{K&u}
Let $K\subset S$ be a monomial ideal and $u$ a monomial in $S$ which
is  regular over $S/K$. If $S/K$ is pretty clean, then $S/(K, u)$ is
pretty clean.
\end{Lemma}

The following proposition is crucial for proving one of the main
results of this section.

\begin{Proposition}\label{J&K}
Let $I\subset S$ be a monomial ideal with
$G(I)=\{u_1,\ldots,u_{m-1},\bar{u}_mx_j^t\}$ where $x_j$ is a free
variable of $I$ and $x_j\nmid \bar u_m$. If
$I_{(\emptyset,\{x_j\})}$ and $I_{(\{x_j\},\emptyset)}$ are pretty
clean, then $I$ is pretty clean.
\end{Proposition}

\begin{proof} We denote $I_{(\emptyset,\{x_j\})}=(u_1,\ldots,u_{m-1},\bar u_m)$
and $I_{(\{x_j\},\emptyset)}=(u_1,\ldots,u_{m-1})$ by $J$ and $K$
respectively.  It is easy to see that $J/I=(I,\bar{u}_m)/I\iso S/(I:
\bar{u}_m)=S/(K,x_j^t)$. Since $S/K$ is pretty clean, by Lemma
\ref{K&u} $J/I$ is also pretty clean. Let $\mathcal F_1\:
I=I_0\subset I_1\subset\cdots\subset I_r=J$ be a pretty clean
filtration of $J/I$ with $I_i/I_{i-1}\iso S/P_i$. Then by
\cite[Corollary 3.4]{HP} $\Supp(\mathcal
F_1)=\Ass(J/I)=\Ass(S/(K,x_j^t))$. Hence $x_j\in P_i$ for
$i=1,\ldots,r$.

By our assumption $S/J$ is pretty clean. Let $\mathcal F_2\:
J=I_r\subset I_{r+1}\subset\cdots\subset I_{r+s}=S$ be a pretty
clean filtration of $S/J$ with $I_{r+i}/I_{r+i-1}\iso S/P_{r+i}$.
Then $P_{r+i}\in\Ass(S/J)$. Hence $x_j\not\in P_{r+i}$ for
$i=1,\ldots, s$.

Combining the prime filtrations $\mathcal F_1$ and $\mathcal F_2$ we
get the prime filtration
\[
\mathcal F\: I=I_0\subset \cdots\subset I_r=J\subset
I_{r+1}\subset\cdots\subset I_{r+s}=S
\]
of $S/I$. Since  $x_j\in P_i$ for $i=1,\ldots,r$ and $x_j\not\in
P_{r+i}$ for $i=1,\ldots, s$, one has $P_i\nsubseteq P_{r+t}$ for
any $i\in [r]$ and any $t\in [s]$. Therefore $\mathcal F$ is a
pretty clean filtration of $S/I$ since $\mathcal F_1$ and $\mathcal
F_2$ are pretty clean filtrations.
\end{proof}

Combining Proposition \ref{J&K} with Lemma \ref{prime}, we get the
following theorem.

\begin{Theorem}
\label{pretty clean} If $I\subset S$ is a monomial ideal of forest
type, then $S/I$ is pretty clean.
\end{Theorem}

\begin{proof}
We use induction on $n$ the number of variables to prove the
assertion. Let $G(I)=\{u_1,\ldots,u_m\}$ and let $x_i$ be a free
vertex of $I$. We may assume that $u_m=\bar u_mx_i^a$ with $a>0$. By
Lemma \ref{prime}, the ideal $J=(u_1,\ldots,u_{m-1},\bar u_m)$ is a
monomial ideal of forest type. It is clear that
$K=(u_1,\ldots,u_{m-1})$ is also a monomial ideal of forest type. By
induction hypothesis  $S/J$ and $S/K$ are pretty clean.  Therefore
by Proposition \ref{J&K}, $S/I$ is pretty clean.
\end{proof}

Let $I\subset S$ be a monomial ideal. According to Stanley
\cite[Section II, 3.9]{St1} and Schenzel \cite{S}, a finite
filtration $\mathcal F\: I=I_0\subset I_1\subset \cdots\subset
I_r=S$ of $S/I$ is called a {\em Cohen--Macaulay filtration} if each
quotient  $I_j/I_{j-1}$ is Cohen--Macaulay, and
\[
\dim(I_1/I_0)<\dim(I_2/I_1)<\cdots<\dim(I_r/I_{r-1}).
\]
The module $S/I$ is called {\em sequentially Cohen--Macaulay} if it
has a Cohen--Macaulay filtration.

It follows from \cite[Corollary 4.3]{HP} that if $S/I$ is pretty
clean, then $S/I$ is sequentially Cohen--Macaulay. Therefore we have
the following corollary, which generalizes the main result of Faridi
\cite{Fa3}.

\begin{Corollary}
\label{more} If $I\subset S$ is a monomial ideal of forest type,
then $S/I$ is sequentially Cohen--Macaulay.
\end{Corollary}

Let $I$ be a monomial ideal, any decomposition of $S/I$  as direct
sum of $K$-vector spaces of the form $uK[Z]$, where $u$ is a
monomial in $S$ and $Z\subset \{x_1,\ldots, x_n\}$, is called a {\em
Stanley decomposition} of $S/I$. Stanley  conjectured in \cite{St}
that there always exists a Stanley decomposition
\[
S/I=\Dirsum_{i=1}^ru_iK[Z_i]
\]
such that $|Z_i|\geq\depth(S/I)$ for all $i\in [r]$. Recently
Stanley's conjecture was studied in several articles, see for
example \cite{HP}, \cite{HSY} and \cite{So}.

In \cite[Theorem 6.5]{HP}, the authors proved the following

\begin{Theorem}
\label{hp} Let $I\subset S$ be a monomial ideal. If $I$ is pretty
clean, then Stanley's conjecture holds for $S/I$.
\end{Theorem}

As an immediate consequence of Theorem \ref{pretty clean} and
Theorem \ref{hp} we have the following:

\begin{Corollary}
\label{forget} If $I$ is a monomial ideal of forest type, then
Stanley's conjecture holds for $S/I$.
\end{Corollary}

Let $\mathcal I$ be the class of monomial ideals with the following
properties:
\begin{enumerate}
\item[(a)] any irreducible monomial ideal is in $\mathcal I$;
\item[(b)] each $I\in \mathcal I$ has a free variable;
\item[(c)] if $x_i$ is a free variable of $I$, then $I\in \mathcal I$
if and only if the minors $I_{(\emptyset, \{x_i\})}$ and
$I_{(\{x_i\}, \emptyset)}$ are in $\mathcal I$.
\end{enumerate}

It is obvious that if a monomial ideal $I$ has the free variable
property, then $I\in \mathcal I$. Moreover we have the following:
\begin{Theorem}
\label{freetree} Let $I\subset S$ be a monomial ideal. The following
statements are equivalent.
\begin{enumerate}
\item[(i)] $I$ is a monomial ideal of forest type;
\item[(ii)] $I$ has the free variable property;
\item[(iii)] $I\in\mathcal I$.
\end{enumerate}
\end{Theorem}

\begin{proof}
(i)\implies (ii):  Let $X_1$ and $X_2$ be any subsets of $X$ with
$X_1\cap X_2=\emptyset$. Since any monomial ideal $J$ with
$G(J)\subset G(I)$ is again a monomial ideal of forest type, this
together with Lemma \ref{prime} imply that $I_{(X_1,X_2)}$ is again
a monomial ideal of forest type. Hence it has a free variable.

(ii)\implies (iii) is obvious.

(iii)\implies (i): We show that $I$ is a monomial ideal of forest
type by using induction on the number of variables $n$ which appear
in $I$. The case $n=1$ is clear. Let $n>1$. Since $I\in\mathcal I$,
we may assume that $G(I)=\{u_1,\ldots,u_m\}$, where $u_m=\bar
u_mx_t^a$ and $x_t$ is a free variable of $I$. Since the ideals
$J=(u_1,\ldots,u_{m-1},\bar u_m)$ and $K=(u_1,\ldots,u_{m-1})$ are
in $\mathcal I$ with less variables, by induction hypothesis $J$ and
$K$ are monomial ideals of forest type. Let $A$ be any subset of
$G(I)$. If $u_m\not\in A$, then $A\subset G(K)$. Hence it has a
leaf. If $u_m\in A$ and $\bar u_m\mid u_j$ for some $u_j\in A$ and
$j\neq m$, then $\gcd(u_m,u_j)=\bar u_m$ and $\gcd(u_m,u_i)\mid \bar
u_m$ for any $i\neq m$. This means that $u_m$ is a leaf of $A$. Now
we may assume that $u_m\in A$ and $\bar u_m\nmid u_j$ for any
 $u_j\in A$ and $j\neq m$. Then $A'=(A\setminus \{u_m\})\cup \{\bar
u_m\}$ is a subset of $G(J)$ and hance it has a leaf. Let $u_p$ be a
leaf of $A'$. Since $x_t$ is a free variable, we have $\gcd
(u_m,u_i)=\gcd(\bar u_m,u_i)$ for any $i\neq m$. If $u_p=\bar u_m$,
then $u_m$ is a leaf of $A$. If $u_p\neq \bar u_m$, then $u_p$
itself is a leaf of $A$.
\end{proof}

A {\em clutter} $\mathcal C$ with vertex set $[n]$ is a family of
subsets of $[n]$, called {\em edges}, with the property that non of
them is contained in another. The edge ideal of a clutter $\mathcal
C$ is defined to be the ideal $I(\mathcal C)= (x_C\: C \text { is an
edge of } \mathcal C)$, where $x_C=\prod_{i\in C} x_i$. A clutter is
a special kind of hypergraph. One may also view a clutter $\mathcal
C$ as the set of facets of some simplicial complex $\Delta$. In this
case, $I(\mathcal C)=I(\Delta)$.

In \cite{TV}, the authors say a clutter $\mathcal C$ has the free
vertex property if the edge ideal $I(\mathcal C)$ has the free
variable property. By Theorem \ref{freetree} one sees that $\mathcal
C$ has the free vertex property if and only if $I(\mathcal C)$ is a
monomial ideal of forest type. If we consider $\mathcal C$ to be the
set of facets of some simplicial complex $\Delta$, then $\mathcal C$
has the free vertex property if and only if $\Delta$ is a forest. In
the following we denote by $\Delta_{\mathcal C}$ the simplicail
complex whose Stanley--Reisner ideal is $I(\mathcal C)$.

As a corollary of Theorem \ref{freetree} and Theorem \ref{pretty
clean}, we obtain the following:

\begin{Corollary}(\cite[Theorem 5.3]{TV})
\label{tv} If the clutter $\mathcal C$ has the free vertex property,
then $S/I(\mathcal C)$ is clean, i.e. $\Delta_{\mathcal C}$ is
shellable.
\end{Corollary}

Let $\mathcal C$ be a clutter and $\Delta$ the simplcial complex
such that $I(\mathcal C)=I(\Delta)$. We say that the clutter
$\mathcal C$ is a forest if $\Delta$ is a simplcial forest. Up to
the order of the vertices and the order of the edges, a clutter is
determined by its incidence matrix and vice versa. The incidence
matrix $M_{\mathcal C}$ is defined as follows: let $1,\ldots,n$ be
the vertices and $C_1,\ldots, C_m$ be the edges of the clutter
$\mathcal C$. Then $M_{\mathcal C}=(e_{ij})$ is an $n\times m$
matrix with $e_{ij}=1$ if $i\in C_j$ and $e_{ij}=0$ if $i\not \in
C_j$. A clutter is called {\em totally balanced} if its incidence
matrix has no square submatrix of order at least $3$ with exactly
two 1's in each row and column. It is known that a totally balanced
clutter has the free vertex property, see \cite[Corollary
83.3a]{Sch}. On the other hand, in \cite[Theorem 3.2]{HHTZ}, it is
shown that $\mathcal C$ is a forest if and only if $\mathcal C$ is
totally balanced. These together with Theorem \ref{freetree} imply
the following:

\begin{Corollary}
\label{balanced} Let $\mathcal C$ be a clutter. The following
statements are equivalent:
\begin{enumerate}
\item[(i)] $\mathcal C$ is a forest;
\item[(ii)] $\mathcal C$ is totally balanced;
\item[(iii)] $\mathcal C$ has the free vertex property.
\end{enumerate}
\end{Corollary}

To end this section, we would like to mention that if $I$ is the
facet ideal of some forest $\Delta$, then $I$ is a monomial ideal of
forest type. Hence $S/I$ is clean. By Corollary \ref{important},
$I^\vee$ has linear quotients.

\section{Some examples and questions}
In Sections 3, we show that the facet ideal $I$ of any forest is
clean and hence Stanley's conjecture holds for $S/I$. There is a
more general class of simplicial complexes, the class of
quasi-forests. It is natural to ask whether the facet ideal of any
quasi-forest is again clean?

According to \cite {Zh}, a connected simplicial complex $\Delta$ is
called a {\em quasi-tree}, if there exists an order $F_1,\ldots,
F_m$ of the facets, such that $F_i$ is a leaf of $\langle
F_1,\ldots, F_i \rangle$ for each $i=1,\ldots, m$. Such an order is
called a {\em leaf order}. A simplicial complex $\Delta$ with the
property that every connected component is a quasi-tree is called a
{\em quasi-forest}. It is clear that any forest is a quasi-forest.

Unfortunately the facet ideal of a quasi-forest need not to be
clean. For example the facet ideal of the quasi-tree
$\Gamma=\langle
\{1,2,3,4\},\{1,4,5\},\{1,2,8\},\{2,3,7\},\{3,4,6\}\rangle$,  as
in Figure \ref{Fig1}, is not clean. Indeed
$$I(\Gamma)^\vee=(x_1x_3,x_2x_4,x_4x_7x_8,x_1x_6x_7,x_1x_4x_7,x_2x_3x_5,x_1x_2x_6,x_2x_5x_6,x_3x_4x_8,x_3x_5x_8)$$
has no linear quotients, even no componentwise linear quotients.

\begin{figure}[hbt]
\begin{center}
\psset{unit=0.5cm}
\begin{pspicture}(-1,-1)(7,7)
%%%%%%%%%%%%%%%%%%%%%%%%%%%%%%%%%%%%%%%%%%%%%%%%%
\rput(0.5,3){2} \rput(3,0.5){3} \rput(-0.5,-0.5){7}
\rput(3,5.5){1} \rput(6.5,6.5){5} \rput(5.5,3){4}
\rput(6.5,-0.5){6} \rput(-0.5,6.5){8}
%%%%%%%%%%%%%%%%%%%%%%%%%%%%%%%%
\pspolygon[style=fyp,fillcolor=medium](1,3)(0,0)(3,1)
\pspolygon[style=fyp, fillcolor=medium](3,1)(6,0)(5,3)
\pspolygon[style=fyp,fillcolor=medium](5,3)(6,6)(3,5)
\pspolygon[style=fyp, fillcolor=medium](3,5)(0,6)(1,3)
\psline(3,5)(3,1) \psline[linestyle=dotted](1,3)(5,3)
\end{pspicture}
\end{center}
\caption{}\label{Fig1}
\end{figure}

One might expect that the facet ideal of any quasi-forest which is
not a forest is not clean. The following example shows that this is
not the case. The facet ideal of the quasi-tree
$\Gamma'=\langle\{1,2,3\},\{2,4,5\},\{2,3,5\},\{3,5,6\}\rangle$, as
in Figure\ref{Fig2}, is clean. Since
$I(\Gamma')^\vee=(x_3x_5,x_2x_5,x_1x_5,x_2x_6,x_2x_3,x_3x_4)$ has
linear quotients in the given order.

\begin{figure}[hbt]
\begin{center}
\psset{unit=1cm}
\begin{pspicture}(0,0.5)(3,3)
%%%%%%%%%%%%%%%%%%%%%%%%%
\rput(1.5,2.6 ){1} \rput(1.5,0.2){5} \rput(0.7,1.4){2}
\rput(2.3,1.4){3} \rput(0.2,0.3){4} \rput(2.8,0.3){6}
%%%%%%%%%%%%%%%%%%%%
\pspolygon[style=fyp,
fillcolor=medium](0.5,0.5)(1.5,0.5)(1,1.4)
 \pspolygon[style=fyp, fillcolor=medium](1.5,0.5)(1,1.4)(2,1.4)
 \pspolygon[style=fyp, fillcolor=medium](1.5,0.5)(2,1.4)(2.5,0.5)
  \pspolygon[style=fyp, fillcolor=medium](1,1.4)(1.5,2.3)(2,1.4)
\end{pspicture}
\end{center}
\caption{}\label{Fig2}
\end{figure}
\medskip

It would be interesting to classify all quasi-forests such that
their facet ideals are clean.

 Even though  $I(\Gamma)$ ($\Gamma$ is the quasi-tree as given in Figure 1) is not
clean we will show that Stanley's conjecture holds for
$S/I(\Gamma)$. First we recall some notation and results from
\cite{HSY}.

Let $\Delta$ be a simplicial complex  on the vertex set $[n]$. A
subset ${\mathfrak I}\subset \Delta$ is called an {\em interval} in
$\Delta$, if there exits faces $F,G\in \Delta$ such that ${\mathfrak
I}=\{H\in \Delta \: F\subseteq H\subseteq G\}$. We denote this
interval by $[F,G]$. A {\em partition} $\mathcal P$ of $\Delta$ is a
presentation of $\Delta$ as a disjoint union of intervals in
$\Delta$. Stanley calls a simplicial complex $\Delta$ {\em
partitionable} if there exists a partition
$\Delta=\Union_{i=1}^r[F_i,G_i]$ with ${\mathcal
F}(\Delta)=\{G_1,\ldots, G_r\}$ and conjectured \cite[Conjecture
2.7]{St1} (see also \cite[Problem 6]{St2}) that each Cohen-Macaulay
simplicial complex is partitionable.  It follows from
\cite[Corollary 3.5 ]{HSY} that the conjecture on Stanley
decompositions implies the conjecture on partitionable simplicial
complexes.

For $F\subseteq [n]$ we set $x_F=\prod_{i\in F}x_i$ and
$Z_F=\{x_i\:\, i\in F\}$. It follows from \cite[Proposition
3.2]{HSY} that if $\Union_{i=1}^r[F_i,G_i]$ is a partition of
$\Delta$, then $\Dirsum_{i=1}^rx_{F_i}K[Z_{G_i}]$ is a Stanley
decomposition of $S/I_\Delta$.

Now let $\Delta$ be  the simplicial complex with the property that
$I_{\Delta}=I(\Gamma)$. Then  Stanley's conjecture holds for
$S/I(\Gamma)$, if there is a partition $\Union_{i=1}^r[F_i,G_i]$ of
$\Delta$  such that $|G_i|\geq \depth S/I_\Delta$ for all $i$.

The facets of $\Delta$ are
\begin{eqnarray*}
&&\{1,3,5,6,7,8\}, \{2,4,5,6,7,8\}, \{1,2,3,5,6\}, \{2,3,4,5,8\}, \{2,3,5,6,8\},\\
&& \{1,4,6,7,8\}, \{3,4,5,7,8\}, \{1,3,4,7,8\},\{1,2,5,6,7\}\,
\text{and}\, \{1,2,4,6,7\}.
\end{eqnarray*}
Consider the  partition
\begin{eqnarray*}
{\mathcal P}&=&[\emptyset,135678]\cup[2,12356]\cup[4,245678]
\cup[14,14678] \cup[27,12567]\cup[34,34578]\\
&&\cup[28,23568]\cup[124,12467]\cup[134,13478]\cup[234,23458]\cup[278,25678]
\end{eqnarray*}
of $\Delta$. Here $135678$ stands for the set $\{1,3,5,6,7,8\}$, and
a similar notation is used for the other sets.

${\mathcal P}$ has the property that the cardinality of upper face
of each interval is greater than or equal to
$$\min \{|F|\: F \text{
is a facet of }\Delta\} \geq \depth(S/I_\Delta)=\depth(S/I(\Gamma)).
$$
This shows that Stanley's conjecture holds for $S/I(\Gamma)$.


\begin{thebibliography}{1}

\bibitem{BH} W.\ Bruns and J.\ Herzog, {\sl Cohen-Macaulay rings}, Cambridge University Press, Cambridge, 1993.

\bibitem{BW} A.\ Bj\"orner, M.\ L.\ Wachs, Shellable nonpure
complexes and posets, I.  Trans. Amer.\ Math.\ Soc. {\bf 348}
(1996), no. 4, 1299--1327.

\bibitem{CH} A.\ Conca and J.\ Herzog, Castelnuovo-Mumford regularity of
products of ideals,  Collect. Math. {\bf 54}  (2003),  no. 2,
137--152.

\bibitem{Dr} A.\ Dress, A new algebraic criterion for shellability,
Beitrage zur Alg.\ und Geom., {\bf 34}(1), (1993), 45--55.

\bibitem{Fa1} S.\ Faridi, The facet ideal of a simplicial complex,  Manuscripta Math. {\bf 109} (2002), 159 -- 174.

\bibitem{Fa2} S.\ Faridi. Cohen-Macaulay properties of square-free monomial
ideals,  J. Combin. Theory Ser. A  {\bf 109}  (2005),  no. 2,
299--329.

\bibitem{Fa3} S.\ Faridi, Simplicial trees are sequentially Cohen-Macaulay,  J. Pure Appl. Algebra
{\bf 190}  (2004),  no. 1-3, 121--136.

\bibitem{HH} J.\ Herzog, T.\ Hibi, Componentwise linear ideals,
Nagoya Mathematical Journal. {\bf 153}, (1999), 141-153.

\bibitem{HHT} J.\ Herzog, T.\ Hibi and N.V.\ Trung, Symbolic powers of monomial ideals and vertex cover
algebras, to appear in Arv. Math.

\bibitem{HHTZ} J. Herzog, T. Hibi, N. Trung and X. Zheng, Standard graded
vertex cover algebras, cycles and leaves, to appear in Trans.\ AMS.

\bibitem{HHZ} J.\ Herzog, T.\  Hibi and X.\ Zheng, Dirac's theorem on chordal
graphs and Alexander duality. European J.\ Comb. 25(7) (2004),
949--960.

\bibitem{Hi} T.\ Hibi, {\sl Algebraic Combinatorics}, Carslaw
Publications, 1992.

\bibitem{HP} J.\ Herzog, D.\ Popescu,  Finite filtrations of modules and
shellable multicomplexes.  Manuscripta Math.  {\bf 121}  (2006),
no. 3, 385--410.
\bibitem{HSY} J.\ Herzog, A.\ Soleyman Jahan, S.\ Yassemi, Stanley decompositions and partitionable simplicial
complexes, to appear in J. Algebraic Combinatorics.

\bibitem{HT} J.\ Herzog, Y.\ Takayama,  Resolutions by mapping cones.
The Roos Festschrift volume, 2.  Homology Homotopy Appl.  {\bf 4}
(2002), no. 2, part 2, 277--294.

\bibitem{K} A.\ Khan, Preety clean module and regular element, in
preperation.

\bibitem{S} P.\ Schenzel, On the dimension filtration and
Cohen-Macaulay filtered modules, Proceed. of the Ferrara meeting in
honour of Mario Fiorentini, ed. F. Van Oystaeyen, Marcel Dekker,
New-York, 1999.



\bibitem{Sch} A.\ Schrijver, {\sl Combinatorial Optimization},
Algorithms and Combinatorics {\bf 24}, Springer-Verlag, Berlin,
2003.

\bibitem{So}  A.\ Soleyman Jahan, Prime filtrations of monomial ideals and
polarizations, to appear in J. Algebra.

\bibitem{St} R.\ P.\ Stanley, Linear Diophantine equations and local cohomology, Invent.\
Math.\ {\bf 68}, (1982), 175--193.

\bibitem{St1} R.\ P.\ Stanley, {\sl Combinatorics and Commutative
Algebra}, Birkh\"auser, 1983.

\bibitem{St2} R.\ P.\ Stanley, Positivity Problems and
Conjectures in Algebraic Combinatorics, In Mathematics: Frontiers
and Perspectives (V. Arnold, M. Atiyah, P. Lax, and B. Mazur,
eds.), American Mathematical Society, Providence, RI, 2000, pp.
295-319.

\bibitem{TV}  A.\ V.\ Tuyl, R.\ H.\ Villarreal,  Shellable graphs and sequentially Cohen-Macaulay
bipartite graphs,  {\it arXiv: math. $CO/0701296$}.

\bibitem{Vi} R.\ H.\ Villarreal, Cohen-Macaulay graphs, Manuscripta
Math.\ {\bf 66}, (1990), 277-293.

\bibitem{Zh} X.\ Zheng, Resolutions of facet ideals. Comm. Algebra {\bf 32}, no. 6, (2004),
2301--2324.
\end{thebibliography}
\end{document}